\documentclass[12pt]{article}
\usepackage[usenames]{color} %used for font color
\usepackage{amssymb} %maths
\usepackage{amsmath} %maths
\usepackage{amssymb,bm,amsmath,amssymb,amsthm}
\usepackage[]{inputenc}
\usepackage[small,nohug,heads=vee]{diagrams}
\diagramstyle[labelstyle=\scriptstyle]
\usepackage{cite}

\usepackage{hyperref}
\DeclareMathOperator{\Gal}{Gal}

\DeclareMathOperator{\Jac}{Jac}
\DeclareMathOperator{\GCD}{gcd}
\DeclareMathOperator{\Gl}{Gl}
\DeclareMathOperator{\Aut}{Aut}
\DeclareMathOperator{\Iso}{Isom}
\DeclareMathOperator{\Pic}{Pic}
\DeclareMathOperator{\id}{id}

\DeclareMathOperator{\Spec}{spec}

\newtheorem{cor}{Corollary}
\newtheorem{ex}{Example}
\newtheorem{rem}{Remark}
\newtheorem{lem}{Lemma}
\newtheorem{defin}{Definition}
\newtheorem{theorem}{Theorem}
\newcommand{\overbar}[1]{\mkern 1.5mu\overline{\mkern-1.5mu#1\mkern-1.5mu}\mkern 1.5mu}

\begin{document}
\title{L-functions of Genus Two Abelian Coverings of Elliptic Curves over Finite Fields}
\author{Pavel Solomatin}
\date{ Leiden, 2016}
\maketitle

\begin{abstract}
Initially motivated by the relations between Anabelian Geometry and Artin's L-functions of the associated Galois-representations, here we study the list of zeta-functions of genus two abelian coverings of elliptic curves over finite fields. Our goal is to provide a complete description of such a list. 
\end{abstract}

\textbf{\\ \\ \\ \\ \\ \\ \\ \\ \\ \\ \\  \\Acknowledgements:} 
First, I would like to thank my advisor, professor Bart de Smit for his invitation to the topic and interesting discussions during the project. I also want to thank professors Bas Edixhoven, Marco Streng, Hendrik Lenstra and Everett Howe for their useful advices.

\newpage
\section{Introduction}

\subsection{Motivation}
The main task of \emph{Grothendieck's anabelian geometry} is to find relations between an object $X$ and its covering space $Y$. Suppose $X=K$ is a number field, then it is natural to consider $Y$ as absolute Galois group $\mathbb G_K = \Gal( \overbar{ K} / K) = \pi_1^{et} (\Spec{K}) $. As a pro-finite group the Galois group $\mathbb G_K$ has a natural topology and it is a topological group. It is a celebrated theorem due to Neukrich and Uchida which states that in this case isomorphism as topological groups of absolute Galois groups $\mathbb G_K =\Gal( \overbar{K} / K) $ and $\mathbb G_L =\Gal( \overbar{ L} / L) $ implies $K \simeq L$. On the other hand there exists an example which shows that it is not enough to consider abelianization $\mathbb G_K^{ab}$ of $\mathbb G_K$ to recover isomorphism class: there are examples of non-isomorphic imaginary quadratic fields with isomorphisms between their abelianized absolute Galois groups, see ~\cite{Peter1}. Nonetheless, it was proven in ~\cite{Quant} by using what authors called \emph{quantum statistical mechanical systems}  that if one has isomorphism of character groups $\widehat{\mathbb G_K^{ab} }\simeq \widehat{ \mathbb G_L^{ab}}$ preserving L-functions, then $K \simeq L$. 
Also, a similar result holds for the function field case, see ~\cite{Quant2}. Moreover, in the number field case Bart de Smit(unpublished) proved that it is possible to construct a single character of order three such that corresponding L-function determines field $K$ up isomorphism.

Motivated by those results we asked ourselves the following informal question: how many Artin's L-functions associated to the representations of $\mathbb G_K^{ab}$  should one know to recover the field up to isomorphism in the function field case?
Here we report about our first result in this direction. 

\subsection{Settings}
Let $k = \mathbb F_q$ be a finite field with $q=p^{n}$, where $p$ is prime. Let $C$ be a curve over $k$ and let $d$ be a natural number prime to $p$. By a curve we always mean smooth projective geometrically connected variety of dimension 1 over $k$. To such a curve one associates the set $\mathbb X_{C}(d,g)$ of all isomorphism classes of smooth projective abelian Galois covers of degree $d$ and genus $g$: 
\begin{defin}
$\mathbb X_{C}(d,g) = \{X $ is an isomorphism class of a curve over $k$, such that $g(X) = g$ and there exists an abelian Galois-covering $\phi : X \to C$, defined over $k$ and of degree $d   \}$. 
\end{defin}

On the function field level, any element $X$ in $\mathbb X_{C}(d,g)$ corresponds to an abelian extension of degree $d$ of the field of functions $k(C)$ of $C$. Let us denote such a Galois group by $G$. According to the standard formalism of Artin's L-functions we have a decomposition law: the ratio of zeta-functions of $X$ and $C$ is equal to the product of all L-functions over all non-trivial characters\footnote{since $G$ is abelian.} of $G$.

Because of the interaction of algebraic geometry and the class-field theory, we have a lot of explicit information about $\mathbb X_{C}(d,g)$. For instance, unramified geometrically connected abelian coverings of $C$ are parametrized by the $\Pic^{0}(C)$ and ramified coverings with fixed ramification divisor $m$ are parametrized by the ray-class group associated to $m$. Later we will discuss it in detail.

Let us consider the set of all zeta-functions $\zeta_{X}(T)$ of curves $X$ in $\mathbb X_C(d,g)$. For any fixed $C, d$ and $g$ this is a finite set of functions. By a famous theorem of A. Weil, they are rational functions of the form $$ \zeta_{X} (T) = \frac{f_{X}(T)}{ (1-T) (1-qT) },$$ where $f_{X}(T) \in \mathbb Z[T]$ is the Weil-polynomial of the covering curve $X$. Such a polynomial keeps a lot of information about $X$, for example we refer reader to the following classical theorem due to Honda and Tate, see ~\cite{Howe1}:  

\begin{theorem}
Let $\Jac(X)$ denote the Jacobian variety of the curve $X$ over $\mathbb F_q$. Let $X'$ denote another curve over $\mathbb F_q$. Then the following are equivalent: 

\begin{enumerate}
	\item $\Jac(X)$ and $\Jac(X')$ are $\mathbb F_q$-isogenous;
	\item The Weil polynomials of $X$ and $X'$ are equal: $f_X(T) = f_{X'}(T)$.
\end{enumerate}

\end{theorem}

In this settings, suppose $X$ is a $\mathbb F_q$-covering of $C$, then we have associated map between Jacobians: $\Jac(X) \to \Jac(C)$ and therefore $\frac{\zeta_{X}(T) }{\zeta_{C}(T)} =\frac{f_X(T) }{f_C(T)}$ is a polynomial with integer coefficients.  In this paper we consider the set $\Lambda_{C}(d,g)$ of all polynomials $\frac{f_{X}(T)}{f_{C}(T)}$ for $X \in \mathbb X_{C}(d,g)$. 

\begin{defin} We define $\Lambda_{C}(d,g) = \{ \frac{f_{X}(T)}{f_{C}(T)} \in \mathbb Z[T]|  X \in \mathbb X_{C}(d,g) \}$.
\end{defin}

Note that, equality $\mathbb {X}_{C}(d,g) = \mathbb{X}_{C'}(d,g)$ obviously implies equality of $\Lambda_{C}(d,g) = \Lambda_{C'}(d,g)$, but not vice-versa.  

It is a remarkable fact that in the case $d=2$ any element in $\Lambda_{C}(2,g)$ is the unique Artin L-function which corresponds to the unique non-trivial representation of the Galois group of fields extension $\mathbb F_q(X)$ over $\mathbb F_q(C)$. This explains the relation with our original motivation.   

Finally here we study $\Lambda_C(d,g)$ where $C = E$ is an elliptic curve and $g=2$. In other words, we study zeta-functions of genus two abelian coverings of elliptic curves. 

\subsection{Results}
In our research we obtain complete information about $\Lambda_{E}(d,2)$ for an elliptic curve $E$ defined over $\mathbb F_q$ with $q=p^n$, $p$ is prime and $p>3$. 

First we state the following corollary of the Galois theory combining with the Riemann-Hurwitz theorem:

\begin{theorem}
For $d>2$ we have $\Lambda_{E}(d,2) = \emptyset$. 
\end{theorem}    

Our main result is the theorem for the case $d=2$. For the sake of shortness here we formulate our result for the case $q=p$. It is turned out that there are two different possibilities: $d=2$ and $d>2$. Before we formulate it we need to introduce some notations. 

Let us denote $a_p = p+1 - \# E(\mathbb F_p)$. Moreover we have the following notations for the sets of polynomials:  

\begin{enumerate}
	\item $A_E =  \{ (pT^2-a'_pT + 1)$,  for $a'_p$ with $ |a'_p| \le 2 \sqrt{p}$, $ a'_p = a_p \bmod (2)  \}$; 
	\item $B_E  = \{ (pT^2-a'_pT + 1)$,  for $a'_p$ with $ |a'_p| \le 2 \sqrt{p}$, $ a'_p = a_p \bmod (4)  \} $.
\end{enumerate}

\begin{theorem}
Assume that $j(E) \ne 0, 1728$. The following holds:
\emph{
\begin{enumerate}
	\item if $E(\mathbb F_p)[2] \not \simeq C_2 \oplus C_2$ then $ \Lambda_{E}(2,2) = A_E $;
	\item	if $E(\mathbb F_p)[2] \simeq C_2 \oplus C_2$  then $ \Lambda_{E}(2,2) = B_E $;
\end{enumerate}
}
\end{theorem}

This theorem says that there are only two sufficiently different cases. Moreover, which case occurs is completely determined by the structure of $\mathbb F_p$-rational 2-torsion points on $E$. The same results hold for curves with $j(E)=0$ or $j(E)=1728$ but with possibly a few exceptions in this list:

\begin{theorem}
Assume that $j(E) = 0$ or $1728$. The following holds:
\emph{
\begin{enumerate}
	\item if $E(\mathbb F_p)[2] \not \simeq C_2 \oplus C_2$ then $ \Lambda_{E}(2,2) \subset A_E $; moreover, the number of elements in the difference does not exceed six: $ |A_E / \Lambda_E(2,2) | \le 6$;
	\item	if $E(\mathbb F_p)[2] \simeq C_2 \oplus C_2$  then $ \Lambda_{E}(2,2) \subset B_E $; moreover, the number of elements in the difference does not exceed six: $ |B_E / \Lambda_E(2,2) | \le 6$;
\end{enumerate}
}
\end{theorem}

During the proof we will provide explicit geometric criteria how to find all possible exceptions.  Also we will explain how to extend those results to the case $q=p^n$, with $n>1$. Roughly speaking for a general field, we also have two different cases depending on the group structure on $E(\mathbb F_q)[2]$, but now we have a little bit more restrictions on possible values of $a'_q$, for details see section \ref{prf}. The proof is based on some classical results concerning geometry of bielliptic curves. More concretely, the main ingredient in our proof is the following result: 
\begin{theorem}
We have a surjective map from the set $\Lambda_E(2,2)$ to the set of pairs $(E', \alpha)$, where $E'$ is an elliptic curve over $\mathbb F_q$ and $\alpha: E[2] \simeq E'[2]$ is an isomorphism between Galois module structure on two-torsion points of $E$ and $E'$, such that $\alpha$ is not the restriction of a geometric isomorphism between $E$ and $E'$.
\end{theorem}

The paper has the following structure: in the next section we show and explain some experimental data for elliptic curves over $\mathbb F_5$. Next we will show how to prove our theorem for $d=2$. Then we will explain cases $d>2$.

\section{Explanations, calculations and examples}
In this section we are going to study the set $\Lambda_{E}(2,2)$ for an elliptic curve $E$ defined over $\mathbb F_q$. Note that any degree 2 covering is actually a Galois covering. Hence, we could use a well-known geometric theory. A good reference here is ~\cite{Kani1}.   
\subsection{Preliminares}
Let $E$ be an elliptic curve over $\mathbb F_q$, with characteristic $p>3$. Let $C$ be a curve of genus $g(C) = 2$  together with the covering map $\phi: C \to E$ of degree~2. Such a curve is called a biellptic curve. 
\begin{ex}
If $E$ is given by the affine equation $y^2 = x^3 + ax + b$, then one could take $C$ with affine part defined by $v^2 = u^6 + au^2 + b$ and map $\phi: (x,y) \to (u^2,v)$.
\end{ex}

Since we have a morphism $\phi$ we have associated map of Jacobian varieties: $  \Jac(E) \to \Jac(C)$. Moreover, because $\dim(\Jac(C))=2$ we have: 
\begin{theorem}
The curve $C$ is bielliptic covering of $E$ if and only if the Jacobian variety $\Jac(C)$ of the curve $C$ is (2,2)-isogenous to a product of two elliptic curves $E \times E'$. 
 \end{theorem}

In the assumptions of the theorem it is not difficult to provide explicit construction of $E'$. Namely, since $C$ is a hyper-elliptic we have a unique involution $\tau \in \Aut(C)$, such that $C / \langle \tau \rangle \simeq \mathbb P^{1}$.    
Since it is unique it lies in the center of $\Aut(C)$. Let us denote by $\sigma$ the element of $\Aut(C)$ such that $C / \langle\sigma\rangle \simeq E$. By our assumption it also has order two. Consider the curve $E' = C/ \langle \sigma \tau \rangle$. Now we have $(\sigma \tau)^2 = \sigma \tau  \sigma  \tau = \sigma^2 \tau^2 = 1$, so we have a degree two map $\phi': C \to E'$. Note that $E' \not \simeq \mathbb P^{1}$, since otherwise we have  $\sigma = \id$. Then, by Riemann-Hurwitz $E'$ is an elliptic curve. Note, that we have the following commutative diagram: 

\begin{diagram}
C   &\rTo   &E' \simeq {C}/ \langle \sigma\tau \rangle \\
\dTo &           &\dTo\\
E \simeq {C}/ \langle \sigma \rangle        &\rTo   &\mathbb P^{1}
\end{diagram}

Finally, we claim that $E \times E'$ is (2,2)-isogenous to the Jacobian surface of $C$. For the proof and complete discussion see ~\cite{Kani2} or ~\cite{Isogn}. 

Now, according to Tate's theorem mentioned in the previous section we have the following relation between Weil polynomials: $$f_C(T) = f_{E}(T) f_{E'}(T) = (qT^2-a_qT + 1)(qT^2-a'_qT + 1) ,$$
where $a_q = q+1 - \# E(\mathbb F_q) $ and $a'_q = q+1 - \# E'(\mathbb F_q) $. So, to describe $\Lambda_E(2,2)$ it is enough to find all possible values of $a'_q$. 

In other words we just proved the following result:
\begin{theorem}
There exists a surjective map from the set $\Lambda_E(2,2)$ to the set of numbers $a'_q$ with property that there exists an elliptic curve $E'$ with $a'_q = q+1-\#E'(\mathbb F_q) $ and with property that abelian surface $E \times E'$ is (2,2)-isogenous to the Jacobian surface of smooth projective curve $C$ defined over $\mathbb F_q$.
\end{theorem}

\subsection{An example over $\mathbb F_5$}
Let us take $q=p=5$. Our task, for any given curve $E$ find all possible values of $a'_5$ as in the above discussion. In order to do that first of all we have to pick a ramification divisor $M$ on $E$ of genus two quadratic cover of $E$. By Riemann-Hurwitz theorem $M$ is of degree two. Then by taking the maximal abelian extension which corresponds to this divisor we obtain a parametrization for all genus two coverings with given ramification data. More concretely from the class-field theory we have the following isomorphism:

$$\phi \colon \Pic_{M}^{0}(E) \to \Gal(F_{M}/F)$$  

Here, $F = \mathbb F_p(E)$ is the function field of $E$, $F_{M}$ is the \emph{Ray class field} corresponding to the pair $(F, M)$ and $Pic_{M}^{0}(E)$ is the ray class group associated to $M$. Hence in order to list all bi-elliptic coverings of $E$ it is enough to list all possible $M$ and for each such $M$ calculate all possible abelian sub-extensions of genus two. By doing that, for any $E$ we provide list of all possible $a'_5$ and compare it with other invariants of $E$. We implement our calculations by using Magma computer algebra system.
Note that $1728 = 3 \mod (5)$ and hence in case $p=5$ we use both values for $j(E)$. Also note that in the table we list isomorphism classes of curves over $k=\mathbb F_5$, not over $\overbar{\mathbb F_5}$.

\begin{table}[]
\centering
\caption{Data for all elliptic curves over $\mathbb F_5$}
\label{my-label}
\begin{tabular}{|l|l|l|l|l|l|}
\hline
Curve $E$            & $j$-invariant & $a_5$ & Values of $a'_5$  & IsSupersingular & $\# \Aut_{k}(E)$ \\ \hline
$y^2 = x^3 + 1$      & 0             & 0     & $0; \pm 2; \pm 4$ & true            & 2         \\ \hline
$y^2 = x^3 + 2$      & 0             & 0     & $0; \pm 2; \pm 4$ & true            & 2         \\ \hline
$y^2 = x^3 + x$      & 3             & 2     & $\pm 2$           & false           & 4         \\ \hline
$y^2 = x^3 + x + 2$  & 1             & 2     & $0; \pm 2; \pm 4$ & false           & 2         \\ \hline
$y^2 = x^3 + x + 1$  & 2             & -3    & $ \pm 1; \pm 3$   & false           & 2         \\ \hline
$y^2 = x^3 + 2x$     & 3             & 4     & $0 ;\pm 2$    & false           & 4         \\ \hline
$y^2 = x^3 + 2x + 1$ & 4             & -1    & $\pm 1; \pm 3$    & false           & 2         \\ \hline
$y^2 = x^3 + 3x$     & 3             & -4    & $0 ; \pm 2$       & false           & 4         \\ \hline
$y^2 = x^3 + 3x + 2$ & 4             & 1     & $\pm 1; \pm 3$    & false           & 2         \\ \hline
$y^2 = x^3 + 4x$     & 3             & -2    & $\pm 2$           & false           & 4         \\ \hline
$y^2 = x^3 + 4x + 1$ & 1             & -2    & $0; \pm 2; \pm 4$ & false           & 2         \\ \hline
$y^2 = x^3 + 4x + 2$ & 2             & 3     & $ \pm 1; \pm 3$   & false           &  2          \\ \hline
\end{tabular}
\end{table}
%%\newpage

\subsection{Observations}

From the data provided by the above table one could note that there exist to different patterns: $a_p$ is odd or even. This is not very difficult to explain: 

\begin{lem}
For any fixed $E$ over $\mathbb F_q$, if $ (qT^2+a'_qT+1) \in \Lambda_E(2,2) $ then  $a_q = a'_q \bmod(2)$.
\end{lem}

\begin{proof}
Consider the covering $\phi: C \to E$ of degree two. From Riemann-Hurwitz theorem we have that ramification divisor of $\phi$ has to be degree two, so it is either a sum of two points of degree one, or a one point of degree two. Here we use the fact that $p>2$ and we don't have so-called wild-ramification. Since $\phi$ is of degree two, we get the number of $\mathbb F_q$-points on $C$ is even. From the decomposition of the Weil polynomial we have $ q+1 - \# C(\mathbb F_q) = q(a_q + a_q')$. Now, just take last equality modulo two and use  the fact that $q$ is odd.
\end{proof}

A second remarkable thing is a some sort of symmetry: if $a'_p$ occurs then also $(-a'_p)$ is in the list. We will explain this phenomena later, but now we note that this is related with quadratic twists of $E$. For proof, see corollary~\ref{Twists}.  

Finally, the last and the main observation is that for \emph{general curve} these are the only restrictions. More contritely, one could note that if $j(E) \ne 0, 1728$ and $E(\mathbb F_p)[2]$ is not isomorphic to the full group $C_2 \oplus C_2$  then any $a'_p = a_p \bmod(2)$ occurs. But if $E(\mathbb F_p)[2] \simeq C_2 \oplus C_2$ then $\Lambda_E(2,2)$ consists of all $a'_p = a_p \bmod(4)$, still provided we are in the case $j(E) \ne 0,1728$. 

Also, the same result holds for $E$ with $j(E)=0, 1728$, but with possible up to $4$ and $6$ exceptions respectively, depending on which twists of $E$ defined over $\mathbb F_p$. Later we will give explicit geometric criteria which answers if there are any exceptions in the list.

\begin{ex}
\begin{enumerate}
	\item Consider the curve $E$ defined by $y^2 = x^3 + x$. In this case $a_5 = 2$, $j(E)=1728$ and it is easy to see that it has four rational two-torsion points: $(0,0)$, $(2,0)$, $(3,0)$ and $\infty$. According to our prediction the only values which may occurs are $\{ \pm 2 \}$. Which is indeed the case.
	\item Consider the curve $E$ defined by $y^2=x^3+1$. Here we have $a_5=0$, $j(E) = 0$ and $E(\mathbb F_5)[2] \simeq C_2$, generated by $(4,0)$. Then we predict that the following values occurs $\{0, \pm 2, \pm 4 \}$. This is coincide with our data.
	\item Consider the curve $E$ defined by $y^2=x^3+3x$. Here we have $a_5=-4$, $j(E)=1728$ and $E(\mathbb F_5)[2] \simeq C_2$, generated by $(0,0)$. But the values $\pm 4$ do not occur in our list. It happens because $j(E)=1728$ and so in this case we have two exceptions.  
\end{enumerate} 
\end{ex}

\subsection{The basic construction}

The crucial fact in our investigation is the following construction due to Kani, see ~\cite{Kani3} and ~\cite{Howe1}. 

Let $n$ be a prime number with $(n,p)=1$. Given two elliptic curves $E$ and $E'$ over $\mathbb F_q$ with isomorphism $\alpha$ as Galois modules $E[n] \simeq E'[n]$, which is anti-isometry with respect to the Weil-paring. Let $\Gamma_{\alpha}$ be the graph of $\alpha$ in $E \times E'$. Consider surface $A_{\alpha} \simeq E \times E' / \Gamma_{\alpha} $. It is $(n,n)$-isogenous to $E \times E'$. Moreover, it turns out that it has \emph{principal polarization} $\theta$ which comes from polarization on $E \times E'$: 

\begin{diagram}
E \times E   &\rTo^{[n]}   & \hat{E} \times \hat{E}   \\
\dTo_{\phi} &           &\uTo_{\hat{\phi}}\\
A _{\alpha}        &\rTo^{\theta}   & \hat{A_{\alpha}}
\end{diagram}

 According to the theorem of A.Weil ~\cite{Weil}: the pair  $(A_{\alpha}, \theta)$ is a polarized Jacobain surface of some, possible not smooth curve $C$ of (arithmetic) genus two.  

\begin{theorem}
The curve $C$ constructed above is smooth if and only if the isomorphism $\alpha$ of Galois modules is not the restriction of a geometric isogeny $\phi$ of degree $d=n(n-i)$ between $E(\bar{k}) \to E'(\bar{k})$, with $0 < i < n$. Moreover, any smooth $C$ such that $\Jac(C)$ is $(n,n)$-isogenous to $E\times E'$ appears in this way. 
\end{theorem} 

In our case $n=2$ and hence $i=1$, but geometric isogeny of degree one is necessary geometric isomorphism, therefore we have:
\begin{cor}
There exists a subjective map $\Lambda_E(2,2)$ to the set of all $a'_q$ such that there exists an elliptic curve $E'$ over $\mathbb F_q$ with $a'_q = q+1-\# E'(\mathbb F_q)$ and an isomorphism $\alpha$ of Galois modules $E[2]$ and $E'[2]$ such that $\alpha$ is not the restriction of a geometric isomorphism between $E$ and $E'$. 
\end{cor}

By working with the Galois module structure on $E[2]$ we provide a proof of our main theorem.

\subsection{On Galois Module Structure on $E[2]$} 

According to the previous section, we must understand which isomorphisms between Galois modules are not restrictions of geometric isomorphisms between curves. In order to do that in this section we briefly recall possible Galois module structures on $E[2]$ and its relations with $\Aut_{\bar k}(E)$. 

Galois group $G_k \simeq \Gal(\bar k / k)$ is generated by the Frobenius element $\pi$. Hence we could restrict our attention to the action of $\pi$ on $E[2]$. Recall that as abelian group $E[2] \simeq \mathbb Z / 2\mathbb Z \oplus \mathbb Z / 2\mathbb Z$. There are three possibilities for the Galois module structure on $E[2]$: either $\pi$ is acting trivially or the action is by a 2-cycle or a 3-cycle. In the first case $E[2]$ has four rational points, in the second case only two and in the later case only one rational point, namely the zero point.

For a given pair of elliptic curves $E, E'$ over $k=\mathbb F_q$ let us consider the set $\Iso_{\bar k}(E,E')$. If it is empty, then $j(E) \ne j(E')$ and any isomorphism between $E[2]$ and $E'[2]$ is not the restriction of a geometric isomorphism. Otherwise, suppose now that $\Iso_{\bar k}(E,E')$ is not empty. Then we have $j(E)=j(E')$ and $|\Iso_{\bar k}(E,E')| = |\Aut_{\bar k}(E)| = |\Aut_{\bar k}(E') | $. Now let $\Iso_{AG}(E[2],E'[2]) $ be the set of isomorphisms between $E[2]$ and $E'[2]$ considered as abelian groups and $\Iso_{G}(E[2],E'[2]) $ be the set of isomorphisms as Galois-modules. 
We have the following:  $$\Iso_{\bar k}(E,E') \to \Iso_{AG}(E[2],E'[2]) \supset \Iso_{G}(E[2],E'[2]) ,$$ where the map is just the restriction of automorphism to the two-torsion points. 

Now we are going to investigate which elements of $\Iso_{G}(E[2],E'[2])$ do not come from restriction of elements of $\Iso_{\bar k}(E,E')$. 

Recall that if $p>3$ then we have exactly the following possibilities: 
\begin{enumerate}
	\item $j(E) \ne 0, 1728$ and $\Aut_{\bar k}(E) = \mathbb Z / 2\mathbb Z$;
	\item $j(E) = 0$ and $E$ is given by $y^2=x^3+b$ and  $\Aut_{\bar k}(E) =\mu_6$;
	\item $j(E) = 1728$ and $E$ is given by $y^2=x^3+ax$ and $\Aut_{\bar k}(E) =\mu_4$.
\end{enumerate}

Therefore, $\# \Iso_{\bar k}(E,E')$ is either 0, 2, 4 or 6. Suppose $\# \Iso_{\bar k}(E,E')$ is not zero and hence we also have a bijective map from $\Iso_{G}(E[2],E'[2])$ to $ \Aut_{G}(E[2]) = \Iso_{G}(E[2],E[2]) $.  Note that there are exactly three types of  $\Aut_{G}(E[2])$: 

\begin{enumerate}
	\item If $E(\mathbb F_q)[2] = C_2 \oplus C_2$, then $\Aut_{G}(E[2]) \simeq \Gl_2(\mathbb F_2)$;  
	\item	If $E(\mathbb F_q)[2] = C_2$, then $\Aut_{G}(E[2]) \simeq C_2$; 
	\item If $E(\mathbb F_q)[2] = \{0 \}$, then $\Aut_{G}(E[2]) \simeq C_3$.
\end{enumerate}

%But then, any such isomorphism acts trivially on two-torsion points and hence there exists exactly one element in $\Iso_{GM}(E[2],E'[2])$ which is the restriction of geometric isomorphism. %Then $\Iso_{\bar k}(E,E') = \{ \pm \phi \}$

\begin{theorem}\label{GalM}
Given two geometrically isomorphic elliptic curves $E$ and $E'$ defined over $\mathbb F_q$, we have that every element of $\Iso_{G}(E[2],E'[2])$ is the restriction of a geometric isomorphism if and only if one of the following pair of conditions holds:
\begin{enumerate} 
	\item	 $j(E)=j(E') = 0$ and $E(\mathbb F_q)[2] = \{0\}$; 
	\item  $j(E)=j(E') = 1728$ and $E(\mathbb F_q)[2] \simeq C_2$.
\end{enumerate}    
\end{theorem}
\begin{proof}
Suppose $j(E)=j(E') \ne 0, 1728$. Let us fix any $\bar k$-isomorphism $\phi: E \to E'$. Then $\Iso_{\bar k}(E,E') = \{ \pm \phi \}$. But then, every element in $\Iso_{\bar k}(E,E')$ acts trivially on two-torsion points and hence there exists at most  one element in $ \Iso_{G}(E[2],E'[2])$ which is the restriction of geometric isomorphism. On the other hand, we always have more than one isomorphism of Galois module structure between $E[2]$ and $E'[2]$.

Suppose $j(E)=j(E')=0$. In this case $E$ can be given by $y^2 = x^3 + b$ and $E'$ is given by $y^2 = x^3 + b'$. Let us fix $t \in \bar k$ such that $t^6 = \frac{b'}{b}$. Consider a map $\phi: E \to E'$ such that $\phi(x,y) = (t^2x, t^3y)$. Let us fix an element $\rho \in \bar k$, $\rho \ne 1$, such that $\rho^{3}=1$. And let us denote by $[\rho]$ the following element of $\Aut(E)$, namely $[\rho](x,y) = (\rho x,y )$. Then $\Iso_{\bar k}(E, E') = \{ \pm \phi,\pm \phi [\rho], \pm \phi [\rho]^2 \}$. By restricting these maps to the maps from $E[2] \to E'[2]$ we obtain three different maps, say $\{1, \tau, \tau^{2} \}$, since as before $\pm$ acts identically on two-torsion points. Now, two torsion points of $E'$ are $\{\infty, (c, 0), (\rho c, 0) , (\rho^2 c, 0 )  \} $, where $c$ is any root of the equation $x^3+b'=0$. Therefore, if $E(\mathbb F_q)[2] \simeq C_2 \oplus C_2$ or $E(\mathbb F_q)[2] \simeq C_2$, then we have an element in $\Iso_{G}(E[2],E'[2])$ which is not the restriction of a geometric isomorphism. But if $E(\mathbb F_q)[2] = \{0\}$, it is easy to see that any element of $\Iso_{G}(E[2],E'[2])$ is indeed the restriction of an element of $\Iso_{\bar k}(E, E')$. 

Finally, suppose we are in the case $j(E)=j(E')=1728$. Then $E$ can be given by $y^2 = x^3 + bx$ and $E'$ is given by $y^2 = x^3 + b'x$. Two-torsion points of $E$ are $\{\infty, (0,0), (\sqrt{-b},0) ,(-\sqrt{-b} , 0) \}$. Note then the point $(0,0)$ is always rational point on both $E$, hence $E(\mathbb F_q)[2]$ is either $C_2$ or $C_2\oplus C_2$. Let us fix an element $i \in \bar k$ such that $i^2 = -1$. We will denote by $[i]$ the following automorphism of $E$: $[i] (x,y) = (-x; iy)$. Let us fix a geometric isomorphism $\phi$ from $E \to E'$. Then $\Iso_{\bar k}(E, E') =  \{ \pm \phi, \pm [i] \phi] \}$. Restriction to two-torsion points gives us two different elements. Therefore, if $E(\mathbb F_q)[2] \simeq C_2$, then any element of $\Iso_{G}(E[2],E'[2])$ is the restriction of an element of $\Iso_{\bar k}$, but if $E(\mathbb F_q)[2] \simeq C_2 \oplus C_2$, then we could pick an isomorphism of Galois module structure on two-torsion points which is not the restriction of a geometric isomorphism.    

\end{proof}

\subsection{The Proof for the case $d=2$}\label{prf}

In this section we give a proof of our main theorem. First, we prove a few auxiliary lemmas.

\begin{lem}
Every quadratic twists $E'$ of an elliptic curve $E$ share isomorphic Galois module structure of two-torsion points and has opposite trace of Frobenius.
\end{lem}
\begin{proof}
For the first statement note that Galois structure of the two-torsion points completely determined by the roots of polynomial $f(x)$, where the elliptic curve $E$ is given by the equation $y^2=f(x)$. Now, one could check that quadratic twist of $E$ is given by $y^2 = d*f(x)$, where $d \in \mathbb F_{q}^{*} / \mathbb F_{q}^{*2}$. Hence $E'[2]$ is isomorphic to the $E[2]$ as Galois module. For the second statement of the proposition note that $\# E(\mathbb F_q) + \# E'(\mathbb F_q) = 2q+2$ and hence $a_q = -a'_q$.
\end{proof}

By using this lemma we obtain the following:
\begin{cor}\label{Twists}
Suppose $pT^2-a'_pT+1$ is in $\Lambda_{E}(2,2)$. Then also $pT^2+a'_pT+1$ is. 
\end{cor}
\begin{proof}
If $-a'_p$ occurs in the list, then there exists an elliptic curve $E_{1}^{'}$ with isomorphism between $E_{1}^{'} [2]$ and $E[2]$, which is not the restriction of a geometric isomorphism between $E_{1}^{'}$ and $E$. Then we could take $E_{2}^{'}$ which is the quadratic twist of $E_{1}^{'}$. It has the same Galois-module structure and negative sign of frobenious. Obviously, we have isomorphism between $E_{2}^{'} [2]$ and $E[2]$, which is not the restriction of a geometric isomorphism between $E_{2}^{'}$ and $E$.    
 \end{proof}

The following result is useful for our purposes. 

\begin{lem}\label{lem4}
$a_q$ is odd if and only if $\pi$ acts as $C_3$ on $E[2]$.
\end{lem}
\begin{proof}
Frobenious element $\pi$ acts on $E[2]$ as three-cycle if and only if it has exactly one fixed point, namely the zero-point. It happens if and only if $E(\mathbb F_q)$ is not divisible by two. But $a_q = q+1-\#E(\mathbb F_q)$, which shows that $a_q$ is even if and only if $\pi$ acts as $C_3$.
\end{proof}

\begin{defin}
Fix a finite field $\mathbb F_q$. Let $N$ be an integer number in the Hasse interval: $N \in [ -2\sqrt{q} ; 2\sqrt{q} ]$. We will call it admissible if there exists an elliptic curve $E$ over $\mathbb F_q$ with $q+1 - \#  E( \mathbb F_q) = N$.  
\end{defin}

The following lemma is the classical statement due to Waterhouse, for reference see ~\cite{Schoof}. 
\begin{theorem} 
The number N is admissible if and only if one of the following conditions holds:
\begin{enumerate}
	\item  $\GCD(p,N)=1$; 
	\item $q=p^{2n+1}$, $n \in \mathbb N$ and one of the following holds: 
		\begin{enumerate}
			\item	N=0;
			\item $N=\pm {2^{n+1}}$ and $p=2$;
			\item $N=\pm {3^{n+1}}$ and $p=3$;
		\end{enumerate}	
	\item $q=p^{2n}$, $n \in \mathbb N$ and one of the following holds: 
		\begin{enumerate}
			\item	$N=\pm 2p^{n}$;
			\item $N=\pm {p^n}$ and $p \ne 1 \mod (3)$;
			\item $N=0$ and $p \ne 1 \mod (4)$;
		\end{enumerate}
\end{enumerate}			
\end{theorem}
  
\begin{rem}Suppose $q=p$ and $p>3$. Then we have $|a'_p| \le 2\sqrt{p} < p$ and hence a condition $\GCD(a'_p ; p) =1$ is automatically holds. Hence, in this settings any number $N$ in the Hasse interval is admissible.
\end{rem}

Combing these results together we already have one cases of our theorem for case $q=p$:
\begin{cor}
Let $E$ be an elliptic curve over $\mathbb F_p$ with $j(E) \ne 0$ and $a_p = 1~\mod (2)$. Then $\Lambda_{E}(2,2)$ consists of all polynomials of the form $ pT^2-a'_pT+1 $, with $a'_p \in [- 2\sqrt{p} ;2\sqrt{p}]$ such that $a'_p=1~\mod (2)$. If $j(E)=0$ the same result holds, with up to 6 exceptions.
\end{cor}
\begin{proof}
Suppose $j(E) \ne 0, 1728$. Given $a'_p $ as above we could construct an elliptic curve $E'$ with $\# E'(\mathbb F_q) = q+1 - a'_p$, by the previous remark. Now, since $a'_p = 1\mod (2)$, by corollary \ref{lem4} there exists isomorphism as Galois-modules between $E[2]$ and $E'[2]$. We have to check that it possible to pick an isomorphism of Galois-modules which is not the restriction of a geometric isomorphism between $E$ and $E'$. This is possible, because of discussion in the theorem $\ref{GalM}$.

If $j(E) = 1728$, then we have at least one rational 2-torsion point, namely $(0,0)$ hence this is not the case.

If $j(E) = 0$ then for any given $a'_p$ we still could pick an elliptic curve $E'$ and find an isomorphism of Galois-module structure. If $j(E')\ne 0$ then any such an isomorphism is not the restriction of a geometric isomorphism. Otherwise if $j(E')=j(E)=0$ then according to the theorem $\ref{GalM}$ in this case any isomorphism between two-torsion parts comes from the restriction of a geometric isomorphism. Obviously, all exceptions which could occur, come from twists of $E$, but there are no more than six twists of elliptic curve defined over $k$.       
\end{proof}

\begin{rem}\emph{Note that even if $E'$ is geometrically isomorphic to $E$, then it \emph{does not} imply that $a'_p$ does not occur in $\Lambda_E(2,2)$, because it may happens that in the isogeny class associated to $a'_p$ there is a curve $E''$ which is not geometrically isomorphic to $E$, but with isomorphism of Galois modules $E[2]$ and $E''[2]$. According to our data this happens very often.}

\end{rem}

\begin{rem}\emph{
There is an obvious generalization to the case $q=p^n$ with $n>1$. Namely, we must pick an \emph{admissible} $a'_q$ with $a'_q = 1\mod (2)$ and take an elliptic curve $E'$. Then, by the same reason there exists an isomorphism of Galois module structure on two-torsion points not coming from a geometric isomorphism between curves, except cases where $j(E')=j(E)=0$.  
}
\end{rem}

The case that $a_q$ is even is a little bit more delicate. The reason is that we have two possibilities for $E (\mathbb F_p ) [2]$. It is either $C_2$ or $C_2 \oplus C_2$.  
 
Namely, suppose we are in the case $a_q = 0 \mod(2)$. Since $q$ is odd, It also means that $\# E(\mathbb F_q) = 0 \mod (2)$. There are two different cases: 

\begin{enumerate}
	\item $\#  E(\mathbb F_q) = 0 \mod (4) $, hence $E(\mathbb F_q)[2] \simeq C_2$ or $E(\mathbb F_q)[2] \simeq C_2 \oplus C_2$;
	\item	$\# E(\mathbb F_q) = 2 \mod (4) $, hence $E(\mathbb F_q)[2] \simeq C_2$;
\end{enumerate}

We see a problem here, because \emph{a priori} given an isogeny class of an elliptic curve $E$ with $\#  E(\mathbb F_q) = 0 \mod (4)$, we can't decide whether there exists curve $E'$ in the same isogeny class with $E'(\mathbb F_q)[2] \simeq C_2$ or with $E'(\mathbb F_q)[2] \simeq C_2 \oplus C_2$. In order to solve this problem, we need two lemmas about two-torsion points on elliptic curves in the isogeny class of given elliptic curve $E$:

\begin{lem}
Suppose $E$ is an elliptic curve over $k=\mathbb F_q$ such that $4 \mid \#E(\mathbb F_q)$. If $E(\mathbb F_q)[2]=C_2$ then there exists an elliptic curve $E'$ defined over $k$ with two properties: 
\begin{enumerate}
	\item $E'$ is $\mathbb F_q$-isogenous to $E$;
	\item $E'(\mathbb F_q)[2]=C_2\oplus C_2$.
\end{enumerate}
\end{lem}
\begin{proof}
Since $4 | \#E(\mathbb F_q)$ and $E(\mathbb F_q)[2]=C_2$ we have that $E(\mathbb F_q)[4]=C_4$. We denote by $P$ a generator of this group. We have $E[2]= \{0, 2P, M, M+2P \}$, where $M$ is a non-rational two-torsion point of $E$. Note that $\pi(M) = M+2P$. Consider $H = \langle2P \rangle$ and elliptic curve $E' = E / \langle H \rangle$. Obviously, $E'$ is isogenous to $E$. We claim that $E'[2] = C_2 \oplus C_2$. Indeed, consider the equation $2R = 2P$, it has exactly four solutions $\{P, 3P, P+M, 3P+M \}$.  Let us denote the map $E \to E/ \langle H \rangle$ by $i$. Then $i(P), i(P+M)$ are two different non-trivial two-torsion points on $E'$. We claim that $\pi$ acts trivially on both $i(P)$ and $i(P+M)$. Indeed, $$i(P) = i(\pi(P)) = \pi i(P)$$ and $$\pi (i(P+M)) = i(P + \pi(M)) = i(P + 2P + M) = i(P+M).$$

But if $\pi$ acts trivially on two non-zero elements of $E'[2]$ then it acts trivially on all points of $E'[2]$.
\end{proof}

Recall, that for any elliptic curve $E$ over $\mathbb F_q$ and prime number $l \ne p$, we associate the Tate module $T_l(E) = \varprojlim_k(E[l^{k}])$. Now, $\pi$ acts on points of $E$ and therefore acts on $T_l(E)$.

\begin{lem}
Suppose $E$ is an elliptic curve over $k=\mathbb F_q$ such that $4 \mid \#E(\mathbb F_q)$. If $E(\mathbb F_q)[2]=C_2\oplus C_2$, then the following are equivalent: 
\begin{enumerate}
	\item There exists an elliptic curve $E'$ with $E'(\mathbb F_q)[2]=C_2$ and  $k$-isogenous to $E$; 
	\item $\pi \in \Aut(T_2(E))$ is not in $\mathbb Z_2^{*}$;
	\item $a_q \ne \pm 2\sqrt{q}$.
\end{enumerate}

\end{lem}
\begin{proof}
First we will prove equivalence between one and two. 

Suppose $\pi$ acts as an 2-adic integer, then any finite 2-subgroup $H$ of $E(\overbar{\mathbb F_q})$ is rational. Now for any $E'$ that is $k$-isogenous to $E$, there exists finite rational subgroup $H \subset E(\overbar{\mathbb F_q})$ such that $E' \simeq E / H $. Let $H_1$ be a maximal group such that $H \subset H_1$ and $H$ is of index two inside $H_1$. Consider $H_1 / H  \subset E/ H  \simeq E'$.  Since $H_1/H$ is a 2-subgroup, then $\pi$ acts trivially on it. On the other hand $E'[2] \simeq H_1/H$, it means that $E'[2]$ is rational.    

Suppose $\pi$ is not in $\mathbb Z_2^{*}$. In means that there exists $P \in T_2(E)$, $P = (P_1, P_2, \dots)$, $P_i \in E[2^{i}]$ such that $\pi(P) \not \in \langle P \rangle$. Since $\pi(P_1) = P_1$, there exists number $i$ such that $\pi(P_i ) \in \langle P_i \rangle = H$, but $\pi(P_{i+1}) \not \in \langle P_{i+1} \rangle =H_1 $. It means that elliptic curve $E' \simeq E/H$, which is isogenous to $E$ has a non-rational two-torsion point, namely $P_{i+1} \mod H$.

Finally we will show that (2) is equivalent to (3). Suppose $\pi$ acts as an element of $\mathbb Z_2^*$, meaning that in some basis of $T_2(E)$ it acts as a scalar matrix. Its characteristic polynomial is $f(x)=x^2-a_qx+q$, which is of the form $f(x)= (x \pm \sqrt{q})^{2}$ when $\pi$ is a scalar. This shows that $a_q = \pm 2\sqrt{q}$. Suppose $a_q = \pm2\sqrt{q}$, then we know that the characteristic polynomial of $\pi$ is $f(x)=(x\pm \sqrt{q})^{2}$. Now we claim that the minimal polynomial of $\pi$ is $ x \pm \sqrt{q}$. Indeed, we have the following sequence: 

\begin{diagram}
E(\overbar{\mathbb F_q})   &\rTo{\pi \pm \sqrt{q}}   &E(\overbar{\mathbb F_q}) &\rTo^{\pi \pm \sqrt{q}} E(\overbar{\mathbb F_q}) 
\end{diagram}

Where the composition of two maps is zero, since the minimal polynomial divides the characteristic polynomial. But, this is the map between two projective curves over algebraically closed field, which means that it is either zero or surjective map. If $(\pi\pm\sqrt{q})$ is not zero map, then also $(\pi\pm\sqrt{q})^{2}$. Therefore the minimal polynomial of $\pi$ is $(x \pm \sqrt{q})$, which means that $\pi$ is a diagonal matrix.

\end{proof}

Combining this two results together we have the following theorem: 

\begin{theorem}\label{Bart}
Given elliptic curve $E$ over $\mathbb F_q$ such that $4 | \# E(\mathbb F_q)$ we have: 
\begin{enumerate}
	\item if $ a_q \ne \pm 2 \sqrt{q}$, then in the isogeny class corresponding to $E$ there exist elliptic curves $E'$, $E''$ with $E'(\mathbb F_q)[2] = C_2$ and $E''(\mathbb F_q)[2] = C_2 \oplus C_2$; 
	\item if $a_q = \pm 2\sqrt{q}$, then any elliptic curve $E'$ isogenous to $E$ has $E'(\mathbb F_q)[2] \simeq C_2 \oplus C_2$. 
\end{enumerate}

\end{theorem}

\begin{cor}\label{sign}
Suppose, $E$ is an elliptic curve with $j(E) \ne 1728$ and with $E (\mathbb F_q)[2] = C_2$. Then $\Lambda_E(2,2)$ consists of all $qT^2-a'_qT+1$ for all admissible $a'_q$ with property $a'_q = 0 \mod(2)$ and $a_q \ne \pm 2\sqrt{q}$. If $j(E) = 1728$ the same result holds with possibly four exceptions.     
\end{cor}
\begin{proof}
Suppose $j(E) \ne 0,1728$. As before, for given admissible $a'_q$ we could construct and elliptic curve $E'$. Condition $a'_q =0 \bmod(2)$ implies that either  $E'(\mathbb F_q)[2] \simeq C_2$ or $E'(\mathbb F_q)[2] \simeq C_2 \oplus C_2$. If $\# E'(\mathbb F_q) = 2 \bmod(4)$ we are done because then $ E'(\mathbb F_q)[2]=C_2$ and theorem $\ref{GalM}$. If $\# E'(\mathbb F_q) = 0 \bmod(4)$, then we are done because of theorem \ref{Bart}.

If $j(E) = 0, 1728$, then we only have problems with $ j(E) = j(E')$, but then theorem $\ref{GalM}$ shows that only possible exceptions  could appear in the case $j(E)=1728$. This exceptions one-to-one correspond to twists of $E$, but there are no more than $4$ twists of an elliptic curve $E$ with $j(E)=1728$.
\end{proof}

\begin{cor}
Suppose, $E$ is an elliptic curve with $E (\mathbb F_q)[2] = C_2 \oplus C_2$. Then $\Lambda_E(2,2)$ consists of all $qT^2-a'_qT+1$ for all admissible $a'_q$ with property $q+1-a'_q = 0 \mod(4)$.    
\end{cor}
\begin{proof}
First note that this condition mentioned above guarantees that for given $a'_q$ there exists an elliptic curve $E''$ in the corresponding isogeny class and as before by theorem $\ref{Bart}$ in this isogeny class we could construct elliptic curve $E'$ with $E'(\mathbb F_q)[2] \simeq C_2 \oplus C_2$. According to theorem $\ref{GalM}$ for any such pair of $E$ and $E'$ we could construct isomorphism between $E[2]$ and $E'[2]$ which is not the restriction of a geometric isomorphism.
\end{proof}

\section{The case $d>2$}
The main purpose of this section is to show:
\begin{theorem}
For $d>2 $ with $p \not{\mid} d$, we have $\Lambda_{E}(d,2) =\emptyset $. 
\end{theorem}
\begin{proof}

We will show, that there is no abelian Galois coverings of an elliptic curve by a genus two smooth projective curve of degree $d>2$, provided that the characteristic of the base field is prime to $d$. Without loss of generality we could suppose $k$ is algebraically closed.

Suppose that $C$ is an abelian covering of $E$ of degree $d>2$. As we already mentioned there exists a unique involution $\tau \in \Aut(C)$ such that $C/\langle\tau \rangle \simeq \mathbb P^{1}$. Moreover, because $\tau$ is unique, it lies in the center of $\Aut(C)$ and hence we have the following commutative diagram: 
\begin{diagram}
C   &\rTo^{2}   &\mathbb P^{1}   \\
\dTo_{d} &           &\dTo_{d}\\
E         &\rTo_{2}   &\mathbb P^{1}
\end{diagram}

Note that \emph{all maps here are abelian Galois coverings}: $C \to E$ is by our assumptions, the shorter morphism $C \to \mathbb P^{1}$ since it has degree 2 and the longer $C \to \mathbb P^1$ is abelian covering because of Galois theory.

Let us apply Riemann-Hurwitz theorem to the covering $C \to E$. We have

$$ (2g_{C} - 2) = d (2g_{E} - 2) + \sum_{p \in C}(e_p -1) ,$$

and hence $$ \sum_{p \in C}(e_p -1) = 2 .$$
Since by assumptions this is a Galois-covering, this means that there are only three possibilities for the ramification divisor: either we have ramification in one point of $E$ of type $(e_1,e_2)=(2,2)$,  two different points on $E$ with ramification index $e_i=2$ or ramification exactly at one point with ramification index $e_1=3$. In the first case we have $d=4$, in the second we have $d=2$ and finally, in the last case we have $d=3$. This proves, that $d \le 4$. Note that if $d=2$ or $d=3$ then the Galois-group of a covering $C \to E$ is cyclic. But if $d=4$ then the Galois group is either $C_4$ or $C_2 \oplus C_2$.  

Now, suppose $d=3$. Consider the map $C \to \mathbb P^{1}$ which is of degree six. Riemann-Hurwitz for this covering tells us :

$$ 2 = 6(-2) +  \sum_{p \in C} (e_p-1),$$
which implies $ \sum (e_p-1) = 14$. Now since we have Galois covering of degree six, the only possible ramification types are $6$, $(3,3)$ and $(2,2,2)$. Suppose we have $m_i$ points of $i$-th ramification type. It implies that $5m_1 + 4m_2 + 3m_1 = 14$, but this equation has only three solutions in non-negative integers: $(2,1,0)$, $(1,0,3)$ and $(0,2,2)$. Riemann-Hurwitz for the covering $C \to E$ gives us: 
$$ (2) = 0 + \sum_{p \in C} (e_p-1),$$  
which implies that the only possible ramification index is $(3)$ with ramification exactly at one point. This excludes possibilities $(2,1,0)$ and $(0,2,2)$ because they both have at least  two points on $C$ with ramification index divisible by $3$. Now, consider the covering $\mathbb P^{1} \to \mathbb P^{1}$ of degree $3$. Riemann-Hurwitz for this case:

$$ -2 = -6 + \sum_{p \in \mathbb P^1} (e_p-1),$$ 
or $4 = \sum (e_p -1)$, which implies that we must have two points with ramification index $(3)$. But then the covering $C \to \mathbb P^{1}$ of degree six must have at least two points with ramification index divisible by three. This provides contradiction to the case $(1,0,3)$ which has only one point with ramification index divisible by three.

The last case is $d=4$. Suppose that the Galois group is $C_2 \oplus C_2$. It implies that there are two different elements $\sigma$, $\tau$ of $\Aut(C)$ each of order two such that there exist two curves $X \simeq C/ \langle \sigma \rangle$ and $Y \simeq C/ \langle \tau \rangle$ and the following commutative diagram: 

\begin{diagram}
C   &\rTo^{2}   &X   \\
\dTo_{2} &           &\dTo_{2}\\
Y         &\rTo_{2}   &E
\end{diagram}

By Riemann-Hurwitz theorem one has $g(X) = g(Y) = 1$ and therefore covering $Y \to E$ is unramified. Hence the covering $C \to X$ is also unramified, which leads to the contradiction.

Finally, suppose that the Galois group is $C_4$. 

 As before, there exist two elements $\sigma, \tau \in \Aut(C)$ such that $$C /\langle\sigma \rangle \simeq E$$ and $$C/\langle\tau\rangle \simeq\mathbb P^1,$$ where $\tau$ has order two and $\sigma$ has order $d=4$.
It implies that there exists an elliptic curve $E' = C / \langle \sigma^{2} \rangle$ such that morphism from $C$ to $E$ factors through $E'$. We denote $G'=\langle\sigma^2 \rangle$, $G=\langle\sigma \rangle$ and $H'' = \langle \tau \rangle$.
  Also we have two subgroups $H = \langle\sigma, \tau\rangle \simeq C_4 \oplus C_2 $, $H' = \langle\sigma^{2}, \tau\rangle \simeq C_2 \oplus C_2$ of $\Aut(C)$ such that $C/H \simeq \mathbb P^{1}$ and $C/H' \simeq \mathbb P^{1}$.   

The following diagram illustrates the whole picture :

\begin{diagram}
&  &C      &&\\
&\ldTo && \rdTo &\\
 C/G' \simeq E' && && C/H'' \simeq \mathbb P^{1}\\
  \dTo &&\rdTo(2,2)       &&\dTo\\
  C/G \simeq E         &&&&C/H' \simeq \mathbb P^{1} \\
&\rdTo &&\ldTo\\
&& C/ H \simeq \mathbb P^{1} 
\end{diagram}

Consider the covering $C \to C/H \simeq \mathbb P^{1}$ of degree eight. Riemann-Hurwitz for this morphism tells us: 

$$ 2 = -16 + \sum_{p \in C} (e_p -1),$$
or equivalently $18 = \sum (e_p -1)$. Since degree of this covering is eight, possible ramification types are $(8)$, $(4,4)$ or $(2,2,2,2)$. Suppose we have $m_i$ points of $i$-th ramification type. Then $7m_1 + 6m_2 + 4m_3 = 18$, which has exactly the following list of solutions in non-negative integers: $(2,0,1)$, $(0,3,0)$, $(0,1,3)$. Riemann-Hurwitz for $C \to E$ gives us $2 = 0 + \sum (e_p-1)$ and therefore we have exactly one ramified point, it has ramification type $(2,2)$. Then solutions $(2,0,1)$ and $(0,3,0)$ are automatically excluded from our consideration. Finally, suppose we are in the case of $(0,1,3)$. We will show that Galois theory implies that there are at least two points with ramification index at least four. Indeed, if $p$ is ramified point for morphism $C/H' \to C/H$, then its inertia group $I_p \subset H \simeq C_4 \oplus C_2$ does not lie in the $H' \simeq C_2 \oplus C_2$. But then, it means it has an element of order at least four. The same time, Riemann-Hurwitz argument shows that there are exactly two points which ramify in the covering $C/H' \to C/H$ and therefore there should be at least two elements of ramification index at least four. 

\end{proof}

\newpage

\bibliography{mybib}{}

\begin{thebibliography}{10}

\bibitem{Peter1}
Athanasios Angelakis and Peter Stevenhagen.
\newblock Imaginary quadratic fields with isomorphic abelian {G}alois groups.
\newblock In {\em A{NTS} {X}---{P}roceedings of the {T}enth {A}lgorithmic
  {N}umber {T}heory {S}ymposium}, volume~1 of {\em Open Book Ser.}, pages
  21--39. Math. Sci. Publ., Berkeley, CA, 2013.

\bibitem{Quant2}
Gunther Cornelissen.
\newblock Curves, dynamical systems, and weighted point counting.
\newblock {\em Proc. Natl. Acad. Sci. USA}, 110(24):9669--9673, 2013.

\bibitem{Quant}
Gunther Cornelissen and Matilde Marcolli.
\newblock Quantum statistical mechanics, {L}-series and anabelian geometry.
\newblock {\em preprint}, 2011.

\bibitem{Isogn}
Kevin~D. Doerksen.
\newblock On the arithmetic of genus two curves with (4,4,)-split jacobians.
\newblock {\em PhD thesis}, 2011.

\bibitem{Kani1}
Gerhard Frey and Ernst Kani.
\newblock Curves of genus 2 with elliptic differentials and associated
  {H}urwitz spaces.
\newblock In {\em Arithmetic, geometry, cryptography and coding theory}, volume
  487 of {\em Contemp. Math.}, pages 33--81. Amer. Math. Soc., Providence, RI,
  2009.

\bibitem{Howe1}
Everett~W. Howe, Enric Nart, and Christophe Ritzenthaler.
\newblock Jacobians in isogeny classes of abelian surfaces over finite fields.
\newblock {\em Ann. Inst. Fourier (Grenoble)}, 59(1):239--289, 2009.

\bibitem{Kani2}
Ernst Kani.
\newblock Elliptic curves on abelian surfaces.
\newblock {\em Manuscripta Math.}, 84(2):199--223, 1994.

\bibitem{Kani3}
Ernst Kani.
\newblock The number of curves of genus two with elliptic differentials.
\newblock {\em J. Reine Angew. Math.}, 485:93--121, 1997.

\bibitem{Schoof}
Ren{\'e} Schoof.
\newblock Nonsingular plane cubic curves over finite fields.
\newblock {\em J. Combin. Theory Ser. A}, 46(2):183--211, 1987.

\bibitem{Weil}
Andr{\'e} Weil.
\newblock Zum {B}eweis des {T}orellischen {S}atzes.
\newblock {\em Nachr. Akad. Wiss. G\"ottingen. Math.-Phys. Kl. IIa.},
  1957:33--53, 1957.

\end{thebibliography}
\bibliographystyle{plain}

\newpage

\tableofcontents

\end{document}